\documentclass{amsart}

\usepackage{amssymb, xcolor,soul}
\makeatletter
\@namedef{subjclassname@2020}{\textup{2020} Mathematics Subject Classification}
\makeatother
\newtheorem{theorem}{Theorem}[section]
\newtheorem{lemma}[theorem]{Lemma}

\newtheorem{corollary}[theorem]{Corollary}

\theoremstyle{definition}
\newtheorem{definition}[theorem]{Definition}

\theoremstyle{remark}

\numberwithin{equation}{section}

\begin{document}

% \title[short text for running head]{full title}
\title[The Borel Ramsey Properties]{The Borel Ramsey Properties for Countable Borel Equivalence Relations}

%    Only \author and \address are required; other information is
%    optional.  Remove any unused author tags.

%    author one information
% \author[short version for running head]{name for top of paper}
\author{Su Gao}
\address{School of Mathematical Sciences and LPMC, Nankai University, Tianjin 300071, P.R. China}
\curraddr{}
\email{sgao@nankai.edu.cn}
\thanks{The authors acknowledge the partial support of their research by the National Natural Science Foundation of China (NSFC) grants 12271263 and 12250710128.}
%
%%    author two information
\author{Ming Xiao}
\address{School of Mathematical Sciences and LPMC, Nankai University, Tianjin 300071, P.R. China}
\curraddr{}
\email{ming.xiao@nankai.edu.cn}
%%\thanks{}

%    \subjclass is required.
\subjclass[2020]{Primary 03E15; Secondary 05C55}

\date{}

%\dedicatory{}

%    "Communicated by" -- provide editor's name; required.
\commby{}

%    Abstract is required.
\begin{abstract} We define some natural notions of strong and weak Borel Ramsey properties for countable Borel equivalence relations and show that they hold for a countable Borel equivalence relation if and only if the equivalence relation is smooth. We also consider some variation of the notion for hyperfinite non-smooth Borel equivalence relations.
\end{abstract}

\maketitle

%    Text of article.

\section{Introduction}
For an integer $n>0$ and a set $X$, let $[X]^n$ denote the set of all $n$-element subsets of $X$. The celebrated infinite Ramsey theorem, proved by Ramsey in his seminal paper \cite{Ramsey1930}, is a cornerstone of infinite combinatorics.

\begin{theorem}[Ramsey] For any integer $n,k>0$ and function $c: [\mathbb{N}]^n\to \{1,\dots, k\}$, there exists $a\in\{1,\dots, k\}$ and an infinite subset $M\subseteq\mathbb{N}$ such that for all $x\in [M]^n$, $c(x)=a$.
	\end{theorem}
	
Here we view $[\mathbb{N}]^n$ as a hypergraph on the set $\mathbb{N}$, which we call the complete hypergraph of dimension $n$ over $\mathbb{N}$; we regard $c$ as a coloring of the hyperedges by $k$ many colors, and $M$ is often referred to as a monochromatic subset. In the arrow notation of Erd\H{o}s and Rado \cite{ER1956}, the infinite Ramsey theorem can be abbreviated as $\aleph_0\rightarrow (\aleph_0)^n_k$.
	
	In this paper, we study generalizations of the infinite Ramsey theorem in the context of invariant descriptive set theory. Specifically, we consider invariant versions of Ramsey-type properties for aperiodic countable equivalence relations. We also impose definability conditions on the equivalence relation, the coloring and the monochromatic set. For basic results of invariant descriptive set theory and undefined terminology in this paper, the reader can consult \cite{Kechris1995} or \cite{Gao2009}.

In our setting, let $X$ be a standard Borel space and let $E$ be an equivalence relation on $X$. Recall that $E$ is called a {\em Borel equivalence relation} if $E$ is a Borel subset of $X^2$, $E$ is {\em countable} if each equivalence class of $E$ is countable, and $E$ is {\em aperiodic} if each equivalence class of $E$ is infinite. Thus, if $E$ is an aperiodic, countable Borel equivalence relation on a Polish space $X$, the quotient space $X/E$ can be viewed as a uniform collection of countably infinite sets on which the Ramsey-type properties can be studied. If the Axiom of Choice is employed, then the infinite Ramsey theorem for this invariant context follows immediately from the classical infinite Ramsey theorem. As usual, if we require that the monochromatic subset be Borel, then the invariant version of the infinite Ramsey theorem can fail. The main objective of this paper is to characterize exactly when these Ramsey-type properties continue to hold in the invariant Borel context.
	
Let $E$ be a countable Borel equivalence relation on a standard Borel space $X$. For an integer $n>0$, let
$$[X]^n_E=\left\{A\in[X]^n\colon \exists x\in X\ ( A\subseteq [x]_E)\right\}, $$
where $[x]_E$ denotes the $E$-equivalence class of $x\in X$. By fixing a canonical linear ordering $<$ of $X$ and enumerating elements of any element of $[X]^n$ in the increasing order of $<$, $[X]^n_E$ can be identified with a Borel subset of $X^n$ and is thus itself a standard Borel space. For an integer $k>0$, we call a Borel function $c:[X]^n_E\to \{1,\dots, k\}$ a {\em Borel $E$-coloring} of {\em dimension} $n$ over $X$ by $k$ many colors.
	
Recall that a subset $Y\subseteq X$ is called a {\em complete section} for $E$ if for all $x\in X$, $Y\cap [x]_E\neq\varnothing$. A complete section $Y$ for $E$ is {\em infinite} if its intersection with every $E$-equivalence class is infinite. Given a Borel $E$-coloring of dimension $n$ over $X$ by $k$ many colors, a complete section $Y\subseteq X$ is {\em $E$-monochromatic} if for every $x\in X$, there is $a\in\{1,\dots, k\}$ such that for all $A\in [Y]^n$ with $A\subseteq [x]_E$, $c(A)=a$. Note that in this definition the color $a$ depends on $[x]_E$, and may vary with it.
	
We define the following Ramsey-type properties in the invariant Borel context.

\begin{definition} Let $X$ be a standard Borel space and let $E$ be a countable Borel equivalence relation on $X$. Let $n,k>0$ be integers.
\begin{enumerate}
\item[(1)] For aperiodic $E$, the {\em (strong) Borel Ramsey property} $E\rightarrow_B (E)^n_k$ means that for any Borel $E$-coloring $c$ of dimension $n$ over $X$ by $k$ many colors, there exists a Borel $E$-monochromatic, infinite complete section $Y\subseteq X$. 
\item[(2)] The {\em weak Borel Ramsey property} $E\rightarrow_B^* (E)^n_k$ means that for any Borel $E$-coloring $c$ of dimension $n$ over $X$ by $k$ many colors, there exists a Borel $E$-monochromatic complete section.
\end{enumerate}
\end{definition}

Both the strong and the weak Borel Ramsey properties for $k=1$ are trivial.  Also, for $n=1$, the Borel Ramsey properties always hold (this will be proved in Lemma~\ref{lem:n=1}).  
For $n, k>1$, it turns out that these notions are closely related to the concept of smoothness in the Borel reducibility hierarchy. Recall that for equivalence relations $E$ and $F$ on standard Borel spaces $X$ and $Y$ respectively, we say that $E$ is {\em Borel reducible} to $F$, denoted $E\leq_B F$, if there is a Borel function $f: X\to Y$ such that for all $x_1, x_2\in X$,
$$ x_1Ex_2\iff f(x_1)Ff(x_2). $$
$E$ is {\em smooth} if it is Borel reducible to the equality relation on a standard Borel space, i.e., there is a standard Borel space $Y$ such that $E\leq_B \ =\upharpoonright\! Y$.

Our main theorem is the following.

\begin{theorem}\label{thm:main} Let $X$ be a standard Borel space and let $E$ be an aperiodic, countable Borel equivalence relation on $X$. The following are equivalent:
\begin{enumerate}
\item[(i)] The Borel Ramsey property $E\rightarrow_B (E)^n_k$ holds for some integers $n, k>1$;
\item[(ii)] The Borel Ramsey property $E\rightarrow_B (E)^n_k$ holds for all integers $n, k>1$;
\item[(iii)] The weak Borel Ramsey property $E\rightarrow^*_B (E)^n_k$ holds for some integers $n,k>1$;
\item[(iv)] The weak Borel Ramsey property $E\rightarrow^*_B (E)^n_k$ holds for all integers $n,k>1$;
\item[(v)] $E$ is smooth.
\end{enumerate}
\end{theorem}

\section{Proof of the Main Theorem}

In this section we prove Theorem~\ref{thm:main}. However, let us first give the promised proof that the Borel Ramsey properties always hold for $n=1$.

\begin{lemma}\label{lem:n=1} Let $X$ be a standard Borel space and $E$ be a countable Borel equivalence relation on $X$. Then for any integer $k>0$, the weak Borel Ramsey property $E\rightarrow^*_B(E)^1_k$ holds. Moreover, if $E$ is aperiodic, then the Borel Ramsey property $E\rightarrow_B(E)^1_k$ holds.
\end{lemma}

\begin{proof} By the Feldman--Moore theorem (c.f., e.g., \cite[Theorem 7.1.4]{Gao2009}), there is a Borel action of a countable group $\Gamma$ on $X$ which generates $E$. Let $\{g_i\}_{i\in\mathbb{N}}$ enumerate the elements of $\Gamma$, with $g_0=1_{\Gamma}$. 

Let $c\colon [X]^1_E\rightarrow\{1,\dots, k\}$ be a Borel $E$-coloring. For each $x\in X$, define 
$$ \chi(x)=\min\left\{c(y)\colon y\in [x]_E\right\}. $$
Then $\chi$ is a Borel $E$-invariant function, i.e., $\chi$ is Borel and for any $xEx'$, $\chi(x)=\chi(x')$. It is obvious that $\chi$ is $E$-invariant. To see that $\chi$ is Borel, note that
$$ \chi(x)=j\iff \exists i\ (c(g_i\cdot x)=j)\wedge \forall j'\ \forall i\ (1\leq j'<j\rightarrow c(g_i\cdot x)\neq j'). $$
Now let 
$$ Y=\{y\in X\colon \chi(y)=c(y)\}. $$
Then $Y$ is a Borel $E$-monochromatic complete section. This proves the weak Borel Ramsey property $E\rightarrow^*_B(E)^1_k$.

Now assume additionally that $E$ is aperiodic. For each $x\in X$, define
$$ \lambda(x)=\min\left\{ c(y)\colon \mbox{$c(z)=c(y)$ for infinitely many $z\in [x]_E$}\right\}. $$
Then again $\lambda$ is a Borel $E$-invariant function. It is obvious that $\lambda$ is $E$-invariant. To see that $\lambda$ is Borel, note that $\lambda(x)=j$ if and only if
$$\forall N\ \exists i>N\ (c(g_i\cdot x)=j)\wedge \forall j'\ [1\leq j'<j\rightarrow \exists N\ \forall i>N\ (c(g_i\cdot x)\neq j')].$$
Now let 
$$Y=\{y\in X\colon \lambda(y)=c(y)\}.$$
Then by the pigeonhole principle, $Y$ is a Borel $E$-monochromatic, infinite complete section.
\end{proof}

As illustrated in the above proof, the use of the Feldman--Moore theorem allows us to express the properties using quantifiers over natural numbers, and thus verifying that the functions and sets are Borel. 

We now start the proof of Theorem~\ref{thm:main} with the following lemma which shows that (v) implies (iv). 

\begin{lemma}\label{prop:half} Let $X$ be a standard Borel space and $E$ be a countable Borel equivalence relation on $X$. If $E$ is smooth, then for any integers $n,k>1$, the weak Borel Ramsey property $E\rightarrow^*_B (E)^n_k$ holds.
\end{lemma}

\begin{proof} By a theorem of Kechris (c.f., e.g. \cite[Theorem~5.4.11]{Gao2009}), any smooth, countable Borel equivalence relation has a Borel transversal. Let $Y\subseteq X$ be a Borel transversal for $E$. Then trivially $Y$ is $E$-monochromatic. 
\end{proof}

The next lemma shows that (v) implies (ii). 

\begin{lemma}\label{lem:half} Let $X$ be a standard Borel space and $E$ be an aperiodic, countable Borel equivalence relation on $X$. If $E$ is smooth, then for any integers $n,k\geq 1$, the Borel Ramsey property $E\rightarrow_B (E)^n_k$ holds.
\end{lemma}

\begin{proof} Let $Y\subseteq X$ be a Borel transversal for $E$. Let $\sigma: X\to Y$ be the function with the graph $\{(x,y)\colon x\in X, y\in Y, xEy\}$. Then $\sigma$ is a Borel selector for $E$.
	
Fix a Borel action of a countable group $\Gamma$ on $X$ which generates $E$. Let $\{g_i\}_{i\in\mathbb{N}}$ enumerate the elements of $\Gamma$, with $g_0=1_{\Gamma}$. 
We define a wellordering $\prec$ of each $E$-class of $X$ as follows. For $xEx'$, define
$$ x\prec x'\iff \exists i\ \left[\, g_i\cdot \sigma(x)=x \wedge \forall j\leq i\ (g_j\cdot \sigma(x)\neq x')\,\right]. $$
$\prec$ is clearly Borel. For each $y\in Y$, $\prec$ wellorders $[y]_E$, with $y$ as the $\prec$-least element of $[y]_E$.	

We then proceed as in the proof of the classical infinite Ramsey theorem for each $[y]_E$, and observe that all quantifiers involved are first-order with respect to natural numbers and so the resulting set is Borel.

The proof is by induction on $n$. The case $n=1$ is given by Lemma~\ref{lem:n=1}. 
		
Now suppose that $c$ is of dimension $n>1$ and for every standard Borel space $X'$, for every smooth, aperiodic, countable Borel equivalence relation $E'$ on $X'$, and for every Borel $E'$-coloring $c'$ of dimension $n-1$ over $X'$, there is a Borel $E'$-monochromatic, infinite complete section for $c'$. Let $X_0=X$ and $Y_0=Y$. Let $c_0$ be the Borel $E$-coloring of dimension $n-1$ over $X\setminus Y$ by $k$ many colors defined by $$c_0(\{x_1,...,x_{n-1}\})=c(\{\sigma(x_1),x_1,...,x_{n-1}\}).$$
Since $Y$ is a transversal and $E$ is aperiodic, $E$ is still aperiodic on $X\setminus Y$.  Let $X_1\subseteq X\setminus Y$ be a Borel $E$-monochromatic, infinite complete section for $c_0$. Let
$$Y_1=\{ x\in X_1\,:\, \mbox{$x$ is the $\prec$-least element of $[x]_E\cap X_1$}\} $$
and
$$ c_1(\{x_1,\dots, x_{n-1}\})=c(\{y, x_1,\dots, x_{n-1}\}) $$
for $\{x_1, \dots, x_{n-1}\}\in [X_1\setminus Y_1]^{n-1}_E$, where $y\in Y_1$ is the unique element with $yEx_1$. Since only one element from each $E$-class is added into $Y_1$, $E$ is still aperiodic when restricted to $X_1\setminus Y_1$.
		
Repeating this process, we obtain infinite sequences $\{X_j\}_{j\in\mathbb{N}}, \{Y_j\}_{j\in\mathbb{N}},\{c_j\}_{j\in\mathbb{N}}$ so that the following hold for all $j\in\mathbb{N}$:
\begin{enumerate}
\item $Y_j\subseteq X_j$ is a Borel transversal of $E$ on $X$;
\item if $y\in Y_j$, $x\in X_j$ and $xEy$, then $y\prec x$;
\item $c_j$ is a Borel $E$-coloring of dimension $n-1$ over $X_j\setminus Y_j$ by $k$ many colors;
\item if $y\in Y_j$, $\{x_1,\dots, x_{n-1}\}\in[X_j\setminus Y_j]^{n-1}_E$ and $yEx_1$, then
    $$c_j(\{x_1,...,x_{n-1}\})=c(\{y,x_1,...,x_{n-1}\});$$
\item $X_{j+1}\subseteq X_j\setminus Y_j$ is a Borel $E$-monochromatic, infinite complete section for $c_j$;
\item if $j'<j$, $y'\in Y_{j'}$, $y\in Y_{j}$ and $y'Ey$, then $y'\prec y$.
\end{enumerate}
Let $Y_{\infty}=\bigcup_{j\in\mathbb{N}}Y_j$. For each fixed $y\in Y_{\infty}$, consider $Y_y=\{z\in Y_\infty\,:\, y\prec z\}$. If we define $c_y:[Y_y]^{n-1}\to \{1,\dots, k\}$ by
$$ c_y(\{x_1,\dots, x_{n-1}\})=c(\{y, x_1,\dots, x_{n-1}\}), $$
then $c_y$ is constant on $[Y_y]^{n-1}$. Let $c''(y)$ be this constant. Then $c''$ is a Borel $E$-coloring of dimension $1$ over $Y_\infty$ by $k$ many colors. By Lemma~\ref{lem:n=1}, we obtain a Borel $E$-monochromatic, infinite complete section $Z\subseteq Y_{\infty}$ for $c''$. One can readily check that $Z$ is in fact a Borel $E$-monochromatic, infinite complete section for $c$.
\end{proof}

The next lemma shows that (iv) implies (ii) and (iii) implies (i). Thus for aperiodic, countable Borel equivalence relations, the strong and the weak Borel Ramsey properties are equivalent.

\begin{lemma}\label{lem:strongweak} Let $X$ be a standard Borel space and $E$ be an aperiodic, countable Borel equivalence relation on $X$. For any integers $n,k>1$, if $E\rightarrow^*_B(E)^n_k$ then $E\rightarrow_B(E)^n_k$.
\end{lemma}
\begin{proof}Fix a Borel action of a countable group $\Gamma$ on $X$ which generates $E$. Let $\{g_i\}_{i\in\mathbb{N}}$ enumerate the elements of $\Gamma$, with $g_0=1_{\Gamma}$.  

Let $c:[X]^n_E\to \{1,\dots,k\}$ be a Borel $E$-coloring. Let $Y\subseteq X$ be a Borel $E$-monochromatic complete section. Let
	$$ Y'=\{y\in Y\colon \mbox{$[y]_E\cap Y$ is finite}\}. $$
Then $Y'$ is Borel since 
$$ y\in Y'\iff \exists N\ \forall m\ \left[\,g_m\cdot y\in Y\rightarrow \exists i\leq N\ (g_m\cdot y=g_i\cdot y)\,\right]. $$ 
If $Y'=\varnothing$ then $Y$ is an infinite complete section and there is nothing to prove. Thus we assume $Y'\neq\varnothing$. Let $X'$ be the $E$-saturation of $Y'$. Then $X'$ is a standard Borel space, and $Y'$ is a Borel complete section of $E\!\upharpoonright\! X'$. %To prove the lemma it suffices to find a Borel $E\!\upharpoonright\! X'$-monochromatic, infinite complete section $Z$. For notational simplicity, assume $Y'=Y$ and $X'=X$, i.e., each $E\!\upharpoonright\! Y$-class is finite. Furthermore, Let $<$ be a Borel linear ordering of $X$. By picking the $<$-least element from each $E\!\upharpoonright\!Y$-class, we may assume that $Y$ is a Borel transversal for $E$, i.e., for each $x\in X$, $[x]_E\cap Y$ is a singleton. Let $\sigma: X\to Y$ be the function with the graph $\{(x,y)\,:\, x\in X, y\in Y, xEy\}$. Then $\sigma$ is a Borel selector for $E$, thus $E$ is smooth, and using Lemma \ref{lem:half} finishes the proof.
Fix a Borel linear ordering $<$ of $X$. Consider
$$ Y''=\{y\in Y\colon \mbox{$y$ is the $<$-least element of $[y]_E\cap Y$}\}. $$
Then $Y''$ is a Borel transversal for $E\!\upharpoonright\!X'$. In particular $E\!\upharpoonright\! X'$ is smooth. By Lemma~\ref{lem:half} we obtain a Borel $E\!\upharpoonright\! X'$-monochromatic, infinite complete section $Z$ on $X'$. Then $Z\cup (Y\setminus Y')$ is a Borel $E$-monochromatic, infinite complete section on $X$.
\end{proof}

In view of Lemma~\ref{lem:strongweak} we will not distinguish between the properties $E\rightarrow_B (E)^n_k$ and $E\rightarrow^*_B(E)^n_k$ for aperiodic, countable Borel equivalence relations $E$ in the remaining discussions. 

\begin{lemma}\label{lem:smallbig} Let $X$ be a standard Borel space and $E$ be an aperiodic, countable Borel equivalence relation on $X$. For any integers $N\geq n>1$ and $K\geq k>1$, if $E\rightarrow_B(E)^N_K$ then $E\rightarrow_B(E)^n_k$.
\end{lemma}

\begin{proof} Suppose $E\rightarrow_B(E)^N_K$. Let $c\colon [X]^n_E\to \{1,\dots, k\}$ be a Borel $E$-coloring. Let $<$ be a Borel linear ordering of $X$. For $\{x_1, \dots, x_N\}\in [X]^N_E$ such that $x_1<\cdots<x_N$, we define
$$ C(\{x_1,\dots, x_N\})=c(\{x_1,\dots, x_n\}). $$
Then $C$ is a Borel $E$-coloring of dimension $N$ over $X$ by $K$ many colors (in fact only $k$ many colors are used). Let $Y$ be a Borel $E$-monochromatic, infinite complete section for $C$. Let
$$ X'=\{x\in X\colon \mbox{there is a $<$-maximum element in $Y\cap [x]_E$}\}. $$
Then $X'$ is a Borel $E$-invariant subset of $X$, and $E\!\upharpoonright\!X'$ is smooth. By Lemma~\ref{lem:half}, there is a Borel $E\!\upharpoonright\!X'$-monochromatic, infinite complete section $Z$ for $c$ on $X'$. We claim that $Y\setminus X'$ is a Borel $E$-monochromatic, infinite complete section for $C$ on $X\setminus X'$. To see this, let $\{x_1,\dots, x_n\}, \{y_1,\dots, y_n\}\in [Y\setminus X']^n_E$ so that $x_1Ey_1$, $x_1<\cdots< x_n$ and $y_1<\cdots<y_n$. By the definition of $X'$, there exist $x_{n+1}, \dots, x_N\in Y\cap[x_1]_E$ and $y_{n+1}, \dots, y_N\in Y\cap [y_1]_E=Y\cap [x_1]_E$ such that
$$ x_n<x_{n+1}<\dots<x_N \mbox{ and } y_n<y_{n+1}<\dots< y_N. $$
Since $Y$ is $E$-monochromatic for $C$, we have that
$$ c(\{x_1,\dots, x_n\})=C(\{x_1,\dots, x_N\})=C(\{y_1,\dots, y_N\})=c(\{y_1,\dots, y_n\}). $$
Thus $Z\cup (Y\setminus X')$ is a Borel $E_0$-monochromatic, infinite complete section for $c$. 
\end{proof}

Next we note that the Borel Ramsey properties are preserved by infinite complete sections.
	
\begin{lemma}\label{lem:completesection} Let $X$ be a standard Borel space and $E$ be an aperiodic, countable Borel equivalence relation on $X$. Let $n,k>1$ be integers. Then the following are equivalent:
\begin{enumerate}			
\item[(a)] The Borel Ramsey property $E\rightarrow_B (E)^n_k$ holds;
\item[(b)] For every Borel infinite complete section $Y\subseteq X$, $F\rightarrow_B(F)^n_k$ holds, where $F=E\!\upharpoonright\! Y$;
\item[(c)] There is a Borel infinite complete section $Y\subseteq X$ so that $F\rightarrow_B (F)^n_k$ holds, where $F=E\!\upharpoonright\! Y$.
\end{enumerate}
\end{lemma}
\begin{proof}
The direction (b)$\Rightarrow$(c) is trivial. To see (c)$\Rightarrow$(a) it suffices to note that a Borel infinite complete section of a Borel infinite complete section is still a Borel infinite complete section.
		
Next we prove (a)$\Rightarrow$(b). Again, by the Feldman--Moore theorem, there is a Borel action of a countable group $\Gamma$ on $X$ which generates $E$. Let $\{g_i\}_{i\in\mathbb{N}}$ enumerate the elements of $\Gamma$, with $g_0=1_{\Gamma}$. Let $Y\subseteq X$ be a Borel infinite complete section and let $F=E\!\upharpoonright\!Y$. For each $x\in X$, let $i_{Y}(x)=\min\{i\colon g_i\cdot x\in Y\}$. Write $g_{Y}(x)=g_{i_Y(x)}(x)$. So $g_{Y}$ is a Borel function with range in $Y$ and is identity on $Y$. Moreover, $g_Y(x)Ex$ for every $x\in X$.
		
Fix a Borel $F$-coloring $c:[Y]_F^n\to\{1,\dots,k\}$. In view of Lemma~\ref{lem:strongweak}, it suffices to find a Borel $F$-monochromatic complete section $Z\subseteq Y$. For this, we extend $c$ to $c'$ on $[X]^n_E$ by defining $$c'(\{x_1,\dots,x_n\})=\left\{\begin{array}{ll}c(\{g_{Y}(x_1),...,g_{Y}(x_n)\})
& \mbox{if $g_Y(x_i)\neq g_Y(x_j)$ whenever $i\neq j$,} \\
1 & \mbox{otherwise.} \end{array}\right.$$
This is clearly a Borel $E$-coloring. By (i), there is a Borel $E$-monochromatic, infinite complete section $W$ for $c'$. Let $Z=\{g_{Y}(x)\colon x\in W\}$. Clearly $Z$ is a Borel complete section of $Y$. To see that $Z$ is $E$-monochromatic, let $\{z_1,\dots, z_n\}\in [Z]^n_E$. There are $x_1,\dots,x_n\in W$ so that $g_{Y}(x_j)=z_j$ for every $1\leq j\leq n$. Since $z_j=g_Y(x_j)$ are pairwise distinct, by definition we have $$c'(\{x_1,\dots,x_n\})=c(\{g_{Y}(x_1),...,g_{Y}(x_n)\})=c(\{z_1,\dots,z_n\}).$$ 
By our choice of $W$, $c'$ is constant for all $\{x_1,\dots,x_n\}\in [W]^n_E$ within any $E$-class, therefore $c$ is constant for all $\{z_1,\dots,z_n\}\in[Z]^n_E$ within any $E$-class. In other words, $Z$ is $E$-monochromatic.			
\end{proof}

Let $E$ and $F$ be equivalence relations on standard Borel spaces $X$ and $Y$ respectively. Recall that $E$ is {\em Borel embeddable} into $F$, denoted $E\sqsubseteq_B F$, if there is an injective Borel function $f: X\to Y$ such that for all $x_1, x_2\in X$, $x_1Ex_2$ iff $f(x_1)Ff(x_2)$.
 
The following is a consequence of Lemma \ref{lem:completesection}.

\begin{lemma}\label{lem:sqsubset} Let $X, Y$ be standard Borel spaces and let $E, F$ be aperiodic, countable Borel equivalence relations on $X, Y$, respectively. Let $n,k>1$ be integers. If $E\sqsubseteq_B F$ and $F\rightarrow_B(F)^n_k$, then $E\rightarrow_B (E)^n_k$.
\end{lemma}

\begin{proof}
	Note that the Borel Ramsey property is preserved by restrictions on invariant Borel subsets and by Borel isomorphisms. Now the lemma follows directly from Lemma \ref{lem:completesection}.
\end{proof}

We are now ready to prove Theorem~\ref{thm:main}. In our proof we will use the equivalence relation $E_0$ defined on the Cantor space $2^\mathbb{N}=\{0,1\}^\mathbb{N}$:
$$ xE_0y\iff \exists n\ \forall m\geq n\ x(m)=y(m). $$
By a theorem of Harrington--Kechris--Louveau, better known as the Glimm-Effros dichotomy theorem (c.f., e.g. \cite[Theorem~6.3.1]{Gao2009}), for any Borel equivalence relation $E$ on a standard Borel space $X$, either $E$ is smooth or else $E_0\sqsubseteq_B E$. In particular, $E_0$ is not smooth.

\begin{proof}[Proof of Theorem~\ref{thm:main}]
In view of Lemmas~\ref{lem:half} and \ref{lem:strongweak}, it suffices to verify the implication (i)$\Rightarrow$(v). For this we show $E_0\not\rightarrow_B (E_0)^2_2$. To see that this is sufficient, note that by Lemma~\ref{lem:smallbig} it follows that $E_0\not\rightarrow_B(E_0)^n_k$ for any $n, k>1$. Now by the Glimm--Effros dichotomy theorem and Lemma~\ref{lem:sqsubset}, if $E$ is nonsmooth, then $E_0\sqsubseteq_B E$ and therefore $E\not\rightarrow (E)^n_k$ for any $n,k>1$. This is the contrapositive of (i)$\Rightarrow$(v).

The rest of the proof is devoted to the construction of a Borel $E_0$-coloring $c$ witnessing $E_0\not\rightarrow_B (E_0)^2_2$.

For any distinct $x,y\in 2^{\mathbb{N}}$ with $xE_0y$, let
$$ D(x,y)=\{i\in\mathbb{N}\colon x(i)\neq y(i)\} $$
and define
$$ c(\{x,y\})=\left\{\begin{array}{ll}1, & \mbox{if $|D(x,y)|$ is even, } \\ 
	2, & \mbox{if $|D(x,y)|$ is odd.}
\end{array}\right.
$$
Then $c$ is Borel since the defining conditions are Borel in $x, y$. We claim that there cannot be a Borel $E_0$-monochromatic, infinite complete section for $c$. To see this, we define a new equivalence relation $F$, which refines $E_0$, by
$$ xFy\iff x=y \mbox{ or } [\,xE_0 y \mbox{ and }c(\{x,y\})=1\,]. $$
It is straightforward to check that this is indeed an equivalence relation. It is also easy to verify that each $E_0$-class consists of exactly two $F$-classes.
In fact, for each $x$, let $\bar{x}$ be obtained from $x$ by flipping its first digit, i.e., $\bar{x}(0)=1-x(0)$ and $\bar{x}(i)=x(i)$ for all $i>0$. Clearly $c(x,\bar{x}$)=2, and for every $yE_0x$, either $yFx$ or $yF\bar{x}$.

Now assume $Y$ is a Borel $E_0$-monochromatic, infinite complete section. Note that if $c$ is constant on a subset of a single $E_0$-class and the constant value is $2$, then this subset must have size at most $2$. It follows that on $[Y]^2_{E_0}$, $c$ always takes the value $1$. This means that $E_0\!\upharpoonright\!Y=F\!\upharpoonright\!Y$.
		
Let $W$ be the $F$-saturation of $Y$. Then $W$ is Borel and $F$-invariant. In addition, letting $\bar{W}=\{\bar{x}:x\in W\}$, we have that $W\cup\bar{W}=2^\mathbb{N}$ and $W\cap\bar{W}=\varnothing$. However, by the first topological 0-1 law (c.f. \cite[Theorem~8.46]{Kechris1995}), $W$ is either meager or comeager as a subset of $2^\mathbb{N}$.  Since the map $x\to \bar{x}$ is a homeomorphism, this is a contradiction. We have shown $E_0\not\rightarrow_B (E_0)^2_2$.
\end{proof}	

Before closing this section we note that clauses (iii)--(v) of Theorem~\ref{thm:main} are equivalent without assuming aperiodicity.

\begin{theorem}\label{thm:main2} Let $X$ be a standard Borel space and let $E$ be a countable Borel equivalence relation on $X$. The following are equivalent:
\begin{enumerate}
\item[(a)] The weak Borel Ramsey property $E\rightarrow^*_B (E)^n_k$ holds for some integers $n,k>1$;
\item[(b)] The weak Borel Ramsey property $E\rightarrow^*_B (E)^n_k$ holds for all integers $n,k>1$;
\item[(c)] $E$ is smooth.
\end{enumerate}
\end{theorem}

\begin{proof} In view of Lemma~\ref{prop:half}, it suffices to show the implication (a)$\Rightarrow$(c). To do this, assume $E\rightarrow^*_B(E)^n_k$ for some $n, k>1$. Define an equivalence relation $F$ on $Y=X\times \mathbb{N}$ by
$$ (x,i)F(y,j)\iff xEy. $$
Then $F$ is an aperiodic, countable Borel equivalence relation. We claim that $F\rightarrow^*_B(F)^n_k$. In fact, let $c\colon [Y]^n_F\to \{1,\dots, k\}$ be a Borel $F$-coloring. Define $c'\colon [X]^n_E\to \{1,\dots, k\}$ by 
$$c'(\{x_1,\dots, x_n\})=c(\{(x_1,0),\dots, (x_n,0)\}). $$
Then $c'$ is a Borel $E$-coloring, and hence there is a Borel $E$-monochromatic complete section $Y\subseteq X$. The set $\{(y,0)\colon y\in Y\}$ is then a Borel $F$-monochromatic complete section.

By Theorem~\ref{thm:main}, $F$ is smooth. Since $E\sqsubseteq_B F$ via the injective map $x\mapsto (x,0)$, $E$ is also smooth.
\end{proof}

\section{Hyperfinite Nonsmooth Equivalence Relations}
 Throughout this section we fix a standard Borel space $X$ and an aperiodic, countable Borel equivalence relation $E$ on $X$. Also fix integers $n,k>1$. As a consequence of our main theorem, the Borel Ramsey property $E\rightarrow_B (E)^n_k$ can be restated as: for any Borel $E$-coloring $c: [X]^n_E\to \{1, \dots, k\}$, there exists a Borel function $S: X\to X^\mathbb{N}$ such that
\begin{enumerate}
\item[(a)] for any $x\in X$ and distinct $i,j\in\mathbb{N}$, $S(x)(i)\neq S(x)(j)$;
 \item[(b)] for any $x\in X$, the set $\{S(x)(i)\colon i\in\mathbb{N}\}$ is a monochromatic subset of $[x]_E$;
 \item[(c)] for any $x, y\in X$, $xEy$ iff $S(x)=S(y)$.
\end{enumerate}
In fact, clauses (a)--(c) together clearly imply that $\{S(x)(i)\colon i\in\mathbb{N}, x\in X\}$ is a Borel $E$-monochromatic, infinite complete section. The set is Borel since
$$ y\in \{S(x)(i)\colon i\in\mathbb{N}, x\in X\}\iff \exists i\in\mathbb{N}\ \exists x\in [y]_E\ [ y=S(x)(i)] $$
and the second quantifier can be turned into a number quantifier by the Feldman--Moore theorem. Conversely, if $Y$ is a Borel $E$-monochromatic, infinite complete section, then from the smoothness of $E$, using the Feldman--Moore theorem, one can inductively build a sequence of pairwise disjoint Borel selectors so as to give a Borel enumeration of the elements of $Y$ on each orbit $[x]_E$, which can then be coded into $S(x)$ with properties (a)--(c).

If we replace (c) by

\begin{enumerate}
\item[(c')] for any $x, y\in X$, if $xEy$ then $\{S(x)(i)\colon i\in \mathbb{N}\}=\{S(y)(i)\colon i\in\mathbb{N}\}$,
\end{enumerate}
the two statements are still equivalent since (a), (b) and (c') together still imply the Borel Ramsey Property. Thus, if we want to consider a meaningful notion of Ramsey-type property for nonsmooth equivalence relations, we have to allow the sets $\{S(x)(i)\colon i\in\mathbb{N}\}$ and $\{S(y)(i)\colon i\in \mathbb{N}\}$ to differ.

In this section we consider hyperfinite nonsmooth equivalence relations. Recall that $E$ is {\em finite} if every $E$-class is finite, and $E$ is {\em hyperfinite} if there is an increasing sequence of Borel finite equivalence relations $\{F_i\}_{i\in\mathbb{N}}$ such that $E=\bigcup_i F_i$. Thus, hyperfinite equivalence relations are necessarily countable Borel equivalence relations. By a theorem of Dougherty--Jackson--Kechris (c.f. \cite{DJK1994} or \cite[Theorem~7.2.3]{Gao2009}), $E$ is hyperfinite iff $E\leq_B E_0$ iff there is a Borel assigment $C\mapsto <_C$ associating with each $E$-class $C$ a linear ordering $<_C$ of $C$ so that there is an order-preserving map from $(C,<_C)$ into $(\mathbb{Z}, <)$. For an aperiodic hyperfinite equivalence relation $E$ we may require that the order type of $(C, <_C)$ is $\zeta$, the order type of $(\mathbb{Z}, <)$, i.e., there is an order isomorphism between $(C, <_C)$ and $(\mathbb{Z}, <)$.

An example of a nonsmooth Borel equivalence relation which is closely related to $E_0$ is the equivalence relation $E_0(X)$, defined on $X^\mathbb{N}$ for any standard Borel space $X$ by
$$ xE_0(X) y\iff \exists i\ \forall j\geq i\ x(j)=y(j). $$
When $X$ is countable, we still have $E_0(X)\leq_B E_0$ and thus $E_0(X)$ is hyperfinite. When $X$ is uncountable, $E_0(X)$ is no longer a countable Borel equivalence relation, but by another theorem of Dougherty--Jackson--Kechris (c.f. \cite{DJK1994} or \cite[Theorem~8.1.5]{Gao2009}), for any countable Borel equivalence relation $E$, if $E\leq_B E_0(X)$, then $E$ is hyperfinite. Following \cite{KL1997}, we denote $E_0(X)$ by $E_1$ when $X$ is an uncountable standard Borel space.

To formulate an appropriate Ramsey-type property for hyperfinite nonsmooth equivalence relations, a natural idea is to allow a small (finite) difference between $S(x)$ and $S(y)$ for $xEy$. This is stated precisely in the following definition.

\begin{definition} Given an $E$-coloring $c: [X]^n_E\to \{1,\dots, k\}$. A function $S:X\to X^{\mathbb{N}}$ is a {\em monochromatic reduction} to $E_1$ for $c$ if:
\begin{enumerate}
\item for any $x\in X$ and distinct $i,j\in\mathbb{N}$, $S(x)(i)\neq S(y)(j)$;
\item for any $x\in X$, the set $\{S(x)(i)\colon i\in\mathbb{N}\}$ is a monochromatic subset of $[x]_E$;
\item for any $x, y\in X$, $xEy$ iff $S(x)E_1 S(y)$.
\end{enumerate}
\end{definition}

However, we note the following negative result.

\begin{lemma} There is a Borel $E_0$-coloring $c$ of dimension $2$ by $2$ colors such that there exists no Borel monochromatic reduction to $E_1$ for $c$.
\end{lemma}

\begin{proof} Consider the $E_0$-coloring $c$ defined in the proof of Theorem~\ref{thm:main}. Let $F$ be the subequivalence relation of $E_0$ also defined in that proof. Assume $S$ is a Borel monochromatic reduction to $E_1$ for $c$. Then for each $x$, $\{S(x)(i)\colon i\in\mathbb{N}\}$ is an infinite monochromatic subset of $[x]_{E_0}$, hence it is included in a single $F$-class. If $xE_0y$, then $S(x)E_1S(y)$, and it follows that $\{S(x)(i)\colon i\in\mathbb{N}\}$ and $\{S(y)(i)\colon i\in\mathbb{N}\}$ are included in a single $F$-class. Now let $Y=\{S(x)(i)\,:\, i\in\mathbb{N},\ x\in X\}$. Then $Y$ is a Borel $E_0$-monochromatic, infinite complete section for $c$, a contradiction.
\end{proof}
	
	To remedy this, we impose additional requirements on the $E$-coloring as in the following definition.
	
\begin{definition}
We call an $E$-coloring $c:[X]^n_E\to\{1,\dots, k\}$ {\em almost transitive} if for any integer $m\geq n$, $A\in [X]^m_E$ and $B_1, B_2\in [A]^{n-1}$, we have
$$ c(\{z\}\cup B_1)=c(\{z\}\cup B_2) $$
for all but finitely many $z\in [A]_E\setminus (B_1\cup B_2)$, where $[A]_E$ is the $E$-saturation of $A$.
\end{definition}

A trivial example of a Borel almost transitive $E$-coloring is the constant coloring. For an example of a nontrivial Borel almost transitive $E$-coloring, consider
a locally finite Borel graph $G$ on a standard Borel space $X$. Let $E_G$ be the equivalence relation given by the connected components of $G$. For distinct $x, y\in X$ with $xE_Gy$, define
$$ c(\{x,y\})=\left\{\begin{array}{ll}1 & \mbox{if there is an edge between $x$ and $y$ in $G$,} \\
 2 & \mbox{otherwise.}
 \end{array}\right.
$$
Then $c$ is a Borel almost transitive $E$-coloring of dimension $2$.

We are now ready to prove a Ramsey-type property for hyperfinite nonsmooth equivalence relations.
	
\begin{theorem}\label{thm:monored} Let $X$ be a standard Borel space and let $E$ be an aperiodic hyperfinite equivalence relation. Let $n, k\geq 1$ be integers. Then for any Borel almost transitive $E$-coloring $c:[X]^n_E\to\{1,\dots, k\}$, there is a Borel monochromatic reduction to $E_1$ for $c$.
\end{theorem}

\begin{proof} Let $\{F_i\}_{i\in\mathbb{N}}$ be an increasing sequence of Borel finite equivalence relations such that $E=\bigcup_{i} F_i$. Without loss of generality assume $F_0$ is the equality relation. Let $C\mapsto <_C$ be a Borel assignment associating to each $E$-class $C$ a linear ordering $<_C$ of $C$. Assume that the order type of $(C, <_C)$ is $\zeta$ for every $E$-class $C$. We use $\prec$ to denote $\bigcup_{x\in X} (<_{[x]_E})$. For $x\in X$, define
$$[x]_{\succ}=\{y\in[x]_E\colon x\prec y\}. $$
		
Let $c:[X]^n_E\to\{1,\dots,k\}$ be a Borel almost transitive $E$-coloring. For each $a\in \{1,\dots, k\}$, we define an $E$-invariant Borel set $X_a$ and a sequence of partial Borel functions $\{f^a_i\}_{i\in\mathbb{N}}$ so that the following hold for any $i\in\mathbb{N}$:%\textcolor{red}{ satisfying following conditions\st{, where each $\mbox{\rm dom}(f^a_i)$ is an $E$-invariant Borel subset of $X$, by induction on $i$ as follows. Let $f^a_0(x)=x$ for all $x\in X$ and $a\in\{1,\dots, k\}$. For $i\geq 0$, assume $f^a_0,\dots, f^a_i$ have been defined so that the following hold}}:
\begin{enumerate}
\item $X=\mbox{\rm dom}(f^a_0)\supseteq \cdots\supseteq \mbox{\rm dom}(f^a_i)$ and each of them is an $E$-invariant Borel subset of $X$;
\item for any $x\in \mbox{\rm dom}(f^a_i)$, if we let $$A=\{f^a_j(y)\colon y\in [x]_{F_i},\ j<i\},$$ then for any $B\in [A]^{n-1}$ with $f^a_i(x)\not\in B$, we have $c(\{f^a_i(x)\}\cup B)=a$;
\item for any $j\leq i$ and any $x, y\in \mbox{\rm dom}(f^a_i)$ with $xF_jy$, we have $f^a_j(x)=f^a_j(y)$;
\item for any $x\in{\rm dom}(f^a_i)$, $xEf^a_i(x)$.
\end{enumerate}
We define $f^a_i$ by induction on $i$. First let $f^a_0(x)=x$ for all $x\in X$ and $a\in\{1,\dots, k\}$. For $i\geq 0$, assume $f^a_0,\dots, f^a_i$ have been defined so that the above conditions hold.
Now for any $x\in \mbox{\rm dom}(f^a_i)$, let 
$$\begin{array}{l}
K^a_x=\{f^a_j(y)\colon y\in[x]_{F_{i+1}},\ j\leq i\}, \\ \\
C^a_x=\{\,z\in[x]_E\colon
\mbox{for all $L\in [K^a_x]^{n-1}$ with $z\not\in L$, $c(\{z\}\cup L)=a$}
\,\},  \mbox{ and}\\ \\
A^a_x=C^a_x\cap \bigcap\left\{[y]_{\succ}\colon y\in K^a_x\right\}.
\end{array} $$
Here $K^a_x$ is the set for condition (2) at the next stage and $C^a_x$ is the set of candidates for the values of $f^a_{i+1}(x)$ which would accomplish condition (2). Now we have two cases according to whether $A^a_x$ is finite. If $A^a_x$ is finite then $f^a_{i+1}(x)$ is undefined. Otherwise, define
$$f^a_{i+1}(x)=\mbox{the $\prec$-least element of } A^a_x.$$

First notice that $\mbox{\rm dom}(f^a_{i+1})$ is indeed Borel since whether $A^a_x$ is finite is a Borel condition. Also all quantifiers involved in construction are inside of $[x]_E$ for every $x$ so by the Feldman--Moore theorem $f^a_{i+1}$ is Borel.

To see that the rest of inductive hypothesis (1) is maintained, take $x,y\in \mbox{\rm dom}(f^a_i)$ so that $xEy$. By the almost transitivity of $c$ and the finiteness of $F_{i+1}$, we have that for all but finitely many $z\in [x]_E$, $z\in C^a_x$ iff $z\in C^a_y$. Thus $A^a_x$ is finite iff $C^a_x$ is $\prec$-bounded, iff $C^a_y$ is $\prec$-bounded, iff $A^a_y$ is finite. Therefore $x\in \mbox{\rm dom}(f^a_{i+1})$ iff $y\in \mbox{\rm dom}(f^a_{i+1})$.
It is easy to see that inductive hypothesis (4) is also maintained by this construction. For inductive hypothesis (2), notice that $K^a_x$ is the set $A$ in condition (2) for step $i+1$ and choosing $f^a_{i+1}(x)$ from $C^a_x$ guarantees the requirement for $f^a_{i+1}$ in condition (2), as it is the set of all possible values which would accomplish condition (2). For inductive hypothesis (3), the only new case is $j=i+1$. Just notice that if $xF_{i+1}y$, then $K^a_x=K^a_y$, from which it follows that $C^a_x=C^a_y$ and $A^a_x=A^a_y$, and hence $f^a_i(x)=f^a_i(y)$. This finishes the inductive definition of $\{f^a_i\}_{i\in\mathbb{N}}$. 

We then define $X_a=\bigcap_i \mbox{\rm dom}(f^a_i)$. 
		
Now, for any $x\in X_a$, $i\geq n$ and $I\in [\{0,\dots, i-1\}]^{n-1}$, we have
$$ c(\{f_j(x)\colon j\in \{i\}\cup I\})=a $$
by condition (1). Also by condition (3) and the fact that the sequence $\{F_i\}_{i\in\mathbb{N}}$ is increasing, we have that if $i\geq j$ and $x, y\in X_a$ with $xF_jy$, then $f^a_i(x)=f^a_i(y)$.
Hence, if we let $S_a(x)(i)=f^a_i(x)$ for $x\in X_a$, then
$S_a\colon X_a\to X_a^\mathbb{N}$ is a Borel monochromatic reduction to $E_1$. Furthermore, all $X_a$ are $E$-invariant. 

Inductively define $X'_a$ for $1\leq a\leq k$ by letting $X'_0=X_0$ and $X'_a=X_a\setminus(\bigcup_{i<a}X'_i)$ for each $1\leq a\leq k$. Then each $X'_a$ is still $E$-invariant, $\bigcup_{1\leq a\leq k}X_a=\bigcup_{1\leq a\leq k}X'_a$ and $X'_a$ are pairwise disjoint. Let $S=\bigcup_{1\leq a\leq k} (S_a\upharpoonright X'_a)$. Then $S$ is a well-defined Borel map as it is a disjoint finite union of Borel maps. For any $x,y\in X$ with $xEy$, $S(x)E_1S(y)$ since for the unique $1\leq a\leq k$ with $x,y\in X'_a$, $S_a(x)E_1S_a(y)$. To see that $S$ is a reduction, just note that $S_a(x)(i)Ex$ for all $i\in\mathbb{N}$ and all $x\in\bigcup_a X_a$. In conclusion, $S$ is a Borel monochromatic reduction from $E$ to $E_1$ for $c$ defined on $\bigcup_a X_a$.

To complete the proof, it suffices to verify that $X\subseteq \bigcup_a X_a$. For this, assume $x\not\in \bigcup_a X_a$. Then for any $a\in \{1,\dots,k\}$, there is $i_a$ such that $x\in \mbox{\rm dom}(f^a_{i_a})\setminus\mbox{\rm dom}(f^a_{i_a+1})$. Let
$$ K=\left\{f^a_j(y)\colon a\in\{1,\dots,k\},\ j\leq i_a,\ y\in[x]_{F_{i_a+1}}\right\}. $$
Then $K$ is finite. By the almost transitivity of $c$, there is a $b\in\{1,\dots, k\}$ such that for all $L\in [K]^{n-1}$, for all but finitely many $z\in [x]_E\setminus L$, $c(\{z\}\cup L)=b$. Thus, at stage $i_b$ of the construction of $X_b$, since $x\in \mbox{\rm dom}(f^b_{i_b})$ and $K^b_x\subseteq K$, we have that $A^b_x$ is infinite. This contradicts $x\not\in \mbox{\rm dom}(f^b_{i_b+1})$.

The proof of the theorem is complete.
\end{proof}

\begin{corollary} Let $X$ be a standard Borel space and $E$ be an aperiodic countable Borel equivalence relation. Let $n, k\geq 1$ be integers. Then $E$ is hyperfinite iff for any Borel almost transitive $E$-coloring $c$ of dimension $n$ over $X$ by $k$ many colors, there is a Borel monochromatic reduction to $E_1$ for $c$. 
\end{corollary}

\begin{proof} We only need to show $(\Leftarrow)$. Since any constant map is a Borel almost transitive $E$-coloring, there is a Borel reduction from $E$ to $E_1$. By a theorem of Kechris--Louveau (\cite[Theorem 1]{KL1997}), every countable Borel equivalence relation which is Borel reducible to $E_1$ is in fact hyperfinite. Thus $E$ is hyperfinite.
\end{proof}

\section*{Acknowledgments}

We would like to thank the anonymous referee for a very careful reading of an earlier version of the paper and for many helpful suggestions. In particular, Lemma~\ref{lem:smallbig}, Theorem~\ref{thm:main2} and its proof are suggested by the referee.

%    Bibliographies can be prepared with BibTeX using amsplain,
%    amsalpha, or (for "historical" overviews) natbib style.
\bibliographystyle{amsplain}
%    Insert the bibliography data here.

\end{document}